\documentclass{amsart}

\usepackage{amsmath,amssymb,amsthm,pinlabel}

\def\co{\colon\thinspace}

\newcommand{\Diff}{\mbox{\rm Diff}}
\newcommand{\Cont}{\mbox{\rm Cont}}
\newcommand{\Aut}{\mbox{\rm Aut}}
\newcommand{\id}{\mbox{\rm id}}
\newcommand{\SO}{\mbox{\rm SO}}
\newcommand{\OO}{\mbox{\rm O}}
\newcommand{\GL}{\mbox{\rm GL}}
\newcommand{\D}{\mathcal{D}}
\newcommand{\N}{\mathbb{N}}
\newcommand{\R}{\mathbb{R}}
\newcommand{\Z}{\mathbb{Z}}

\newcommand{\bfx}{\mathbf{x}}
\newcommand{\tb}{{\tt tb}}

\newcommand{\oK}{\overline{K}}
\newcommand{\ophi}{\overline{\phi}}

\newcommand{\rot}{{\tt rot}}
\newcommand{\xist}{\xi_{\mathrm{st}}}
\newcommand{\rc}{r_{\mathrm{c}}}

\newtheorem{thm}{Theorem}
\newtheorem{lem}[thm]{Lemma}
\newtheorem{prop}[thm]{Proposition}

\theoremstyle{definition}
\newtheorem*{rem}{Remark}
\newtheorem*{rems}{Remarks}

\newtheorem*{ack}{Acknowledgements}

\begin{document}

\title[The diffeotopy group of $S^1\times S^2$]{The diffeotopy group
of $S^1\times S^2$\\via contact topology}

\author{Fan Ding}
\address{Department of Mathematics, Peking University,
Beijing 100871, P.~R. China}
\email{dingfan@math.pku.edu.cn}
\author{Hansj\"org Geiges}
\address{Mathematisches Institut, Universit\"at zu K\"oln,
Weyertal 86--90, 50931 K\"oln, Germany}
\email{geiges@math.uni-koeln.de}

\date{9 April 2009}

\begin{abstract}
As shown by H. Gluck in 1962, the diffeotopy group of
$S^1\times S^2$ is isomorphic to $\Z_2\oplus\Z_2\oplus\Z_2$.
Here an alternative proof of this result is given, relying
on contact topology. We then discuss two applications to
contact topology: (i) it is shown that the fundamental group of the space of
contact structures on $S^1\times S^2$, based at the standard
tight contact structure, is isomorphic to~$\Z$;
(ii) inspired by previous work of M.~Fraser,
an example is given of an integer family 
of Legendrian knots in $S^1\times S^2\# S^1\times S^2$
(with its standard tight contact structure) that can be distinguished
with the help of contact surgery, but not by the classical invariants
(topological knot type, Thurston--Bennequin invariant, and rotation number).
\end{abstract}


\maketitle

\section{Introduction}
The {\em diffeotopy group} $\D (M)$ of a smooth manifold $M$ is the quotient
of the diffeomorphism group $\Diff (M)$ by its normal subgroup
$\Diff_0(M)$ of diffeomorphisms isotopic to the identity. Alternatively,
one may think of the diffeotopy group as the group $\pi_0(\Diff (M))$
of path components of $\Diff (M)$, since any continuous path
in $\Diff (M)$ can be approximated by a smooth one, i.e.\ an isotopy.
We use this terminology to emphasise that we work in the differentiable
category throughout. In the topological realm, with diffeomorphisms
replaced by homeomorphisms, one speaks of the {\em homeotopy group}.
In either situation, the more popular term is {\em mapping class group},
sometimes with the attribute `extended' in order to indicate that
orientation-reversing maps are allowed.

Quite a bit is known about the diffeotopy groups of $3$-manifolds.
The theorem of Cerf~\cite{cerf68} says that $\D (S^3)=\Z_2$.
The diffeotopy groups of lens spaces were computed independently
by Bonahon~\cite{bona83} and Hodgson--Rubinstein~\cite{horu85}. For
other known results, open questions, and an extensive bibliography,
see Kirby's problem list~\cite{kirb97}, especially Problems 3.34 to 3.36.

The diffeotopy group of $S^1\times S^2$ was determined by
Gluck~\cite{gluc62}. He showed that $\D (S^1\times S^2)\cong
\Z_2\oplus \Z_2\oplus \Z_2$. (Actually, Gluck dealt with
the homeotopy group, but in dimension $3$ this amounts to the same;
see the discussion in Section 5.8 of~\cite{horu85}.)
Our aim in this note is to derive that result by contact topological
means. The main ingredients are the classification
of contact structures on $S^1\times S^2$ up to isotopy, a result of
Colin about isotopies of $2$-spheres in contact $3$-manifolds, and
a theorem of Giroux concerning the space of contact elements on $\R^2$
and its contactomorphism group. This may indicate to what extent one might
hope to generalise our method to other $3$-manifolds.

The key point in the determination of the diffeotopy group of
$S^1\times S^2$ is to show that any diffeomorphism acting trivially
on homology is isotopic to either the identity or a diffeomorphism
$r$ of order~$2$ (up to isotopy) that will be described in the next section.
This argument will take up Sections \ref{section:contacto}
to~\ref{section:reduction}.

We then put this result to use in contact topology.
In Section~\ref{section:space} we show that
the fundamental group of the space of
contact structures on $S^1\times S^2$, based at the standard
tight contact structure~$\xist$, is isomorphic to~$\Z$. This follows essentially
from the observation that the mentioned diffeomorphism $r$
is isotopic to a contactomorphism $\rc$ of infinite
order in the contactomorphism group (as was noticed previously
by Gompf~\cite{gomp98}).

In Section~\ref{section:knots} we give an explicit description of
an infinite family of homologically trivial
Legendrian knots in $(S^1\times S^2\#
S^1\times S^2,\xist\#\xist )$, all of which have the same
topological knot type, Thurston--Bennequin
invariant, and rotation number, but which are pairwise not Legendrian
isotopic. This family has previously been described by Fraser, albeit
in an implicit fashion only. Moreover, we shall
explain why we regard her argument as incomplete.
\section{The diffeotopy group of $S^1\times S^2$}
Given a manifold ~$M$, write $\Aut_i(M)$ for the group of automorphisms of
the homology group $H_i(M)$. We consider the homomorphism
\[ \begin{array}{rccrcl}
\Phi\co & \D (S^1\times S^2) & \longrightarrow & \Aut_1(S^1\times S^2) &
                 \oplus & \Aut_2(S^1\times S^2) \\
        & [f]                & \longmapsto     & (f_*|_{H_1}           &
                 ,      & f_*|_{H_2}).
\end{array} \]
Since $H_i(S^1\times S^2)\cong\Z$ for $i=1,2$, this gives
a homomorphism $\Phi\co\D (S^1\times S^2)\rightarrow\Z_2\oplus\Z_2$.
For the interpretation of $\Z_2$ as the automorphism group of $\Z$ it is
convenient to write $\Z_2$ multiplicatively with elements $\pm 1$.
In order to study the properties of~$\Phi$, we introduce the
following diffeomorphisms.

Write $r_{\theta}$ for the rotation of $S^2\subset\R^3$ about the $x_3$-axis
through an angle~$\theta$. We think of $S^1$ as $\R /2\pi\Z$.
Define diffeomorphisms $s,a,r$ of $S^1\times S^2$ by
\begin{eqnarray*}
s(\theta ,\bfx ) & = & (-\theta ,\bfx ),\\
a(\theta ,\bfx ) & = & (\theta ,-\bfx ),\\
r(\theta ,\bfx ) & = & (\theta ,r_{\theta}(\bfx )).
\end{eqnarray*}
Then $\Phi ([s])=(-1,1)$, $\Phi ([a])=(1,-1)$, and $\Phi ([r])=(1,1)$.
So $\Phi$ is surjective, and  --- since $s$ and $a$ commute with each
other --- a splitting of $\Phi$ can be
defined by sending $(1,1)$ to $[\id_{S^1\times S^2}]$, the element
$(-1,1)$ to~$[s]$, and $(1,-1)$ to~$[a]$. We therefore have a split short
exact sequence
\[ \ker\Phi\rightarrowtail\D (S^1\times S^2)\twoheadrightarrow
\Z_2\oplus\Z_2.\]

\begin{lem}
The class $[r]$ has order $2$ in $\D (S^1\times S^2)$.
\end{lem}

\begin{proof}
The fact that the order of $[r]$ is at most $2$ follows
from $r^2(\theta ,\bfx )=(\theta ,r_{2\theta}(\bfx ))$ and
$\pi_1(\SO_3)=\Z_2$. Actually, this shows that $r^2$ is isotopic to
the identity via an isotopy preserving the $S^2$-leaves
in the product foliation of $S^1\times S^2$.

In order to show that $r$ is not isotopic to the identity, we choose
a trivialisation of the tangent bundle $T(S^1\times S^2)$ by
an oriented frame.
We may assume that along $S^1\equiv S^1\times\{ (0,0,1)\}$ this frame is
$\partial_{\theta},\partial_{x_1},\partial_{x_2}$. Then any
orientation-preserving diffeomorphism $f$ of $S^1\times S^2$ induces an
element $[Tf|_{S^1}]\in\pi_1(\GL_3^+)$.
Isotopic diffeomorphisms induce the same element. The identity
on $S^1\times S^2$ induces the trivial element; the diffeomorphism $r$,
the non-trivial element in $\pi_1(\GL_3^+)=\Z_2$.
\end{proof}

Our main goal will be to prove the following statement.

\begin{prop}
\label{prop:kernel}
The subgroup $\ker\Phi\subset\D (S^1\times S^2)$ is generated by $[r]$,
and hence isomorphic to~$\Z_2$. In other words, any diffeomorphism
of $S^1\times S^2$ acting trivially on homology is isotopic to
either {\rm id} or~$r$.
\end{prop}

The result $\D (S^1\times S^2)\cong\Z_2\oplus\Z_2
\oplus\Z_2$ is an immediate consequence: from the split short exact sequence
above we know that $\D (S^1\times S^2)$ is the semidirect product of
the normal subgroup $\Z_2$ and the quotient $\Z_2\oplus\Z_2$;
but a normal subgroup of order $2$ is central, so the action of the
quotient by conjugation is trivial.

Thus, let $f$ be a diffeomorphism of $S^1\times S^2$ acting trivially on
homology. (In particular, $f$ preserves the orientation.)
The strategy will be to isotope $f$ step by step to a
diffeomorphism satisfying a number of additional properties, until
we arrive at id or~$r$. After each step, we continue to write $f$ for the
new diffeomorphism.
\section{From a diffeomorphism to a contactomorphism}
\label{section:contacto}
We shall rely freely on some fundamental notions and results from contact
topology, all of which can be found in~\cite{geig08}.

The standard tight contact structure
$\xist$ on $S^1\times S^2\subset S^1\times\R^3$ is given by
\[ \alpha:=x_3\, d\theta +x_1\, dx_2-x_2\, dx_1 =0.\]
This is the unique positive tight contact structure on $S^1\times S^2$ up to
isotopy, see ~\cite[Thm.~4.10.1]{geig08}.

\begin{lem}
The diffeomorphism $f$ is isotopic to a
contactomorphism\footnote{Contact structures are assumed to be
cooriented; contactomorphisms are understood to
preserve the coorientation.}
of~$\xist$.
\end{lem}

\begin{proof}
The contact structure
$Tf(\xist )$, which is again positive and tight, is isotopic
to~$\xist$. Gray stability~\cite[Thm.~2.2.2]{geig08} then gives the
desired isotopy.
\end{proof}

Later on we shall need a contactomorphism $\rc$ representing
the class~$[r]$ (there should be no confusion with the
notation $r_{\theta}$ used earlier).
There are two ways of exhibiting such a contactomorphism:
the first one uses the above description of $(S^1\times S^2,\xist )$;
the second one is better adapted to describing the effect on
Legendrian curves in the front projection picture.

A straightforward computation yields
\[ r^*\alpha = (x_3+x_1^2+x_2^2)\, d\theta +x_1\, dx_2-x_2\, dx_1.\]
We claim that $r$ can be isotoped to a contactomorphism $\rc$
by an isotopy that shifts each $2$-sphere $\{\theta\}\times S^2$
along its characteristic foliation induced by~$\xist$. Indeed, that foliation
is given by the vector field
\[ X=x_1x_3\partial_{x_1}+x_2x_3\partial_{x_2}+(x_3^2-1)\partial_{x_3},\]
with singular points at $(x_1,x_2,x_3)=(0,0,\pm 1)$. One computes
\[ L_X\alpha = i_Xd\alpha = (x_3^2-1)\, d\theta +2x_1x_3\, dx_2
-2x_2x_3\, dx_1 = 2x_3\alpha-(1+x_3^2)\, d\theta.\]
This shows that the flow of $X$ has the desired effect of decreasing
the $d\theta$-component relative to the $dx_1$- and $dx_2$-components;
thus, a suitable rescaling of $X$ by a function that depends only
on $x_3$ will give us a flow that moves the
contact structure $\ker (r^*\alpha )$ back to $\ker\alpha =\xist$.

An alternative picture, due to Gompf~\cite[p.~636]{gomp98},
is based on a contactomorphism
\[ \bigl( S^1\times (S^2\setminus\{\mbox{\rm poles}\} ),\xist \bigr)\cong
\bigl(\R\times (\R /2\pi\Z )^2,\ker (dz+x\, dy )\bigr) .\]
The $2$-spheres $\{\theta_0 \}\times S^2$
correspond to the annuli $\{ y=y_0\}$, each compactified by two points
at $x=\pm\infty$. Now a simple description of a contactomorphism
$\rc$ in the class $[r]$ is given by a Dehn twist along a circle
$\{ y=y_0\}$ in the torus $(\R /2\pi\Z )^2$ (plus a shift in the
$x$-direction to make it a contactomorphism):
\[ (x,y,z)\longmapsto (x-1,y,y+z).\]
Figure~\ref{figure:contact-r} shows the
effect of that Dehn twist on the Legendrian circle $t\mapsto (0,t,0)$ in
$\R\times (\R /2\pi\Z )^2$, followed by a Legendrian isotopy corresponding
to Gompf's `move~6', of which we shall see more in the next section.
(The figure shows the front projection to the $yz$-torus.)

\begin{figure}[h]
\labellist
\small\hair 2pt
\pinlabel $\cong$ at 198 74
\pinlabel $\rc$ at 198 50
\endlabellist
\centering
\includegraphics[scale=0.4]{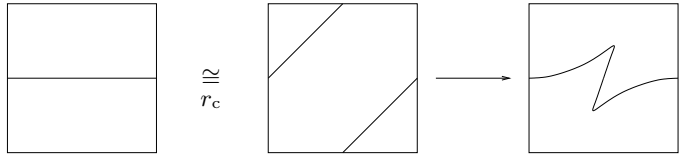}
  \caption{The contactomorphism $\rc$, followed by `move 6'.}
  \label{figure:contact-r}
\end{figure}
\section{Fixing a $2$-sphere}
As described in \cite[Section~2]{gomp98} or \cite[Section 11.1]{gost99},
cf.~\cite{dige09}, one can represent the contact manifold
$(S^1\times S^2,\xist )$ by the Kirby diagram with one $1$-handle only
(Figure~\ref{figure:1-handle}) in the standard contact structure
on~$S^3$. The attaching balls for the $1$-handle are drawn
as round balls, but it is understood that these balls are
in fact chosen in such a way that the characteristic foliation
on their boundary induced by the standard contact structure
on $S^3$ is the same as the characteristic foliation on
$\{\theta\}\times S^2$ induced by~$\xist$.

\begin{figure}[h]
\centering
\includegraphics[scale=0.25]{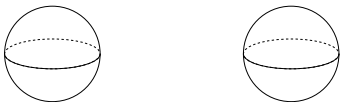}
  \caption{$S^1\times S^2$ with its standard tight contact structure.}
  \label{figure:1-handle}
\end{figure}

\begin{lem}
The contactomorphism $f$ is contact isotopic to a contactomorphism
fixing a sphere $S_0:=\{ 0\}\times S^2$, which we think of as
the boundary of the attaching balls in Figure~\ref{figure:1-handle}.
\end{lem}

\begin{proof}
Since $f$ is a contactomorphism, $\xist$ induces the same
characteristic foliation on $f(S_0)$ as on~$S_0$. As shown
by Colin~\cite{coli97}, with $\xist$ being tight this implies that
$f(S_0)$ and $S_0$ are contact isotopic --- and hence $f$ contact isotopic
to a contactomorphism fixing~$S_0$ ---, provided the two $2$-spheres
are topologically isotopic.

For showing the existence of such a topological isotopy, we essentially
rely on stage one of Gluck's proof~\cite{gluc62}; the argument is included
here for the reader's convenience.

If $S_1:=f(S_0)$ is disjoint from $S_0$, then those two spheres bound
a compact manifold that constitutes an $h$-cobordism $W$ between them.
This follows from the fact that the two spheres are homotopic; recall
that $f$ acts trivially on $H_2(S^1\times S^2)=\pi_2(S^1\times S^2)$.
This $h$-cobordism $W$ is contained inside $S^1\times S^2$ and therefore
does not contain any fake $3$-cells --- without appeal to Perelman's
positive answer to the Poincar\'e conjecture. It follows that $W$ is
diffeomorphic to $S^2\times [0,1]$, and hence $S_0$ isotopic to~$S_1$.
(This argument is due to Laudenbach~\cite{laud73}.)

In the general case, we first
use an isotopy to bring $S_0$ and $S_1$ into general position,
such that they intersect transversely in a finite number of circles.
We want to isotope $S_1$ further to a sphere disjoint from~$S_0$;
as just explained this will conclude the argument.

Let $C$ be one of the circles of intersection, chosen in such a way
that it bounds a $2$-disc $D_1$ in~$S_1$ not containing
any other circles of intersection. In $S_0$, the circle
$C$ bounds two $2$-discs $D_0$ and $D_0'$. One of the
$2$-spheres $D_0\cup D_1$ and $D_0'\cup D_1$, say the former,
bounds a $3$-ball, as can be seen by considering the situation
in the universal cover of $S^1\times S^2$. This allows us to isotope
$S_1$ across this $3$-ball in order to remove the circle $C$ of
intersection. In the process, all circles of intersection
contained in $D_0$ will be removed as well. See
Figure~\ref{figure:remove} for a schematic picture. Iterating this
procedure, we separate $S_0$ and~$S_1$.
\end{proof}

\begin{figure}[h]
\labellist
\small\hair 2pt
\pinlabel $D_0$ [bl] at 170 93
\pinlabel $D_1$ [bl] at 228 100
\pinlabel $S_1$ [br] at 86 90
\pinlabel $S_0$ [b] at 145 145
\pinlabel {$S^1\times S^2$} [b] at 248 145
\endlabellist
\centering
\includegraphics[scale=0.5]{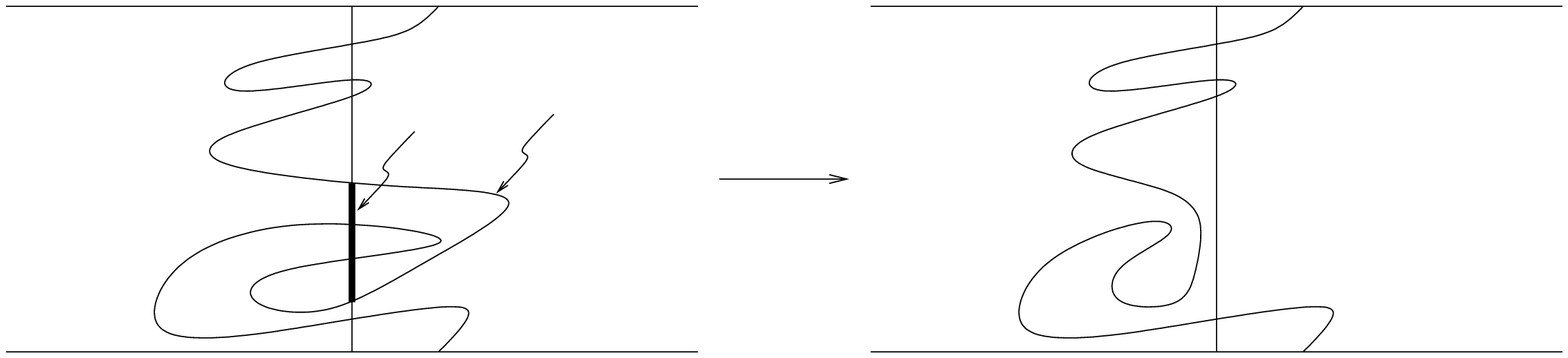}
  \caption{Removing intersections between $S_0$ and $S_1$.}
  \label{figure:remove}
\end{figure}
\section{Fixing a Legendrian circle}
\label{section:Legcircle}
Now consider the oriented Legendrian circle $K_0\subset (S^1\times S^2,\xist )$
representing what we shall call the {\em positive}
generator of $H_1(S^1\times S^2)$, as shown in the
front projection picture in Figure~\ref{figure:1-handle-K}. (It is
understood that $\R^3\subset S^3$ be equipped with the standard contact
structure $dz+x\, dy=0$; Legendrian knots are illustrated in
the front projection to the $yz$-plane.) This corresponds to the
Legendrian circle on the left-hand side in
Figure~\ref{figure:contact-r}. In particular,
that figure shows that the Legendrian knot $\rc (K_0)$ is Legendrian isotopic
to the positive stabilisation $S_+K_0$ of~$K_0$.

\begin{figure}[h]
\centering
\includegraphics[scale=0.25]{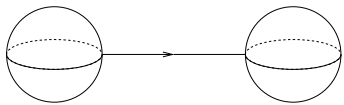}
  \caption{The Legendrian circle $K_0$.}
  \label{figure:1-handle-K}
\end{figure}

\begin{lem}
For some $k\in\Z$, the contactomorphism $\rc^k\circ f$ is contact
isotopic to a contactomorphism fixing~$K_0$.
\end{lem}

\begin{proof}
The image $f(K_0)$ will be some Legendrian knot representing
the positive generator of $H_1(S^1\times S^2)$ and, since $f$ fixes~$S_0$,
going exactly once over the $1$-handle. With the help of
`move~6' from~\cite{gomp98}, or what is also called the light bulb
trick (cf.~\cite{dige09}), one can unknot $f(K_0)$
via a Legendrian isotopy (which extends to a contact isotopy
by~\cite[Thm.~2.6.2]{geig08}). An example is shown in
Figure~\ref{figure:lightbulb1}, where the final result of the isotopy
is actually~$K_0$. In general, the result will be some
(multiple) stabilisation of~$K_0$.

\begin{figure}[h]
\centering
\includegraphics[scale=0.45]{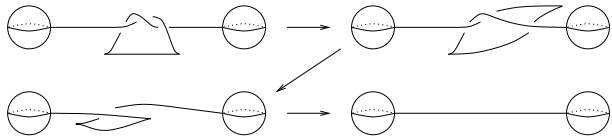}
  \caption{The light bulb trick used for unknotting.}
  \label{figure:lightbulb1}
\end{figure}

Here is a more `algorithmic' description of this unknotting procedure.
First of all, by \cite[Thm.~2.2]{gomp98} we may assume that,
after a Legendrian isotopy, $f(K_0)$ is in standard form, i.e.\
its front projection is contained entirely between the two attaching
balls for the $1$-handle. Given a knotted piece of string with loose
ends, one can clearly unknot it by contracting the string from one of its
ends. If we imagine the attaching balls as the ends of such a piece of
string, this contraction can be regarded as a motion of the right-hand ball,
say. We thus remove all crossings in the front projection, while
preserving the cusps; we need Legendrian Reidemeister moves
of the second kind to slide all the cusps adjacent to the
right-hand attaching ball over or under another strand in order to
remove the crossing with that strand. With the light bulb trick this
translates into a Legendrian isotopy with the attaching balls fixed.
The final result of this Legendrian isotopy will be a Legendrian knot
whose front projection has no crossings, but which now winds several times
around the right-hand attaching ball in the $yz$-plane. One can bring
the knot back into standard form (and still no crossings in the
front projection) as follows: perform a move of type~6 to introduce a single
kink in the front projection; then remove the kink with a Legendrian
Reidemeister move of the first kind (see Figure~\ref{figure:kink});
each such move reduces the (absolute) winding number of the front projection
of $f(K_0)$ around the right-hand ball.

\begin{figure}[h]
\centering
\includegraphics[scale=0.45]{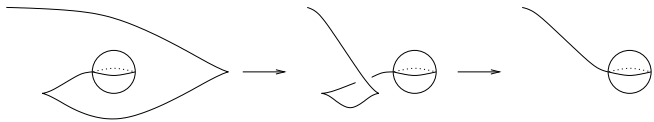}
  \caption{Reducing the winding number.}
  \label{figure:kink}
\end{figure}

Positive and negative stabilisations can then be removed in pairs by
a further application of the light bulb trick, as shown in
Figure~\ref{figure:lightbulb2}.

\begin{figure}[h]
\centering
\includegraphics[scale=0.45]{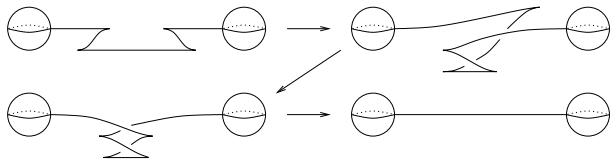}
  \caption{The light bulb trick used for removing stabilisations.}
  \label{figure:lightbulb2}
\end{figure}

Thus, $f$ is contact isotopic to a contactomorphism that maps $K_0$
to a stabilisation $S_{\pm}^nK_0$ for some $n\in\N_0$.
Then $\rc^{\mp n}\circ f$ is contact isotopic to a contactomorphism
that fixes~$K_0$.
\end{proof}

\begin{rem}
Even powers of $\rc$ are isotopic to the identity, but not contact
isotopic to the identity. This follows from the observation that
the application of $\rc$ to $K_0$ increases its rotation number by~$1$.
Notice that $\xist$ is trivial as a $2$-plane bundle; a global
non-vanishing section of $\xist$ is given by
\[ (x_2-x_1)\partial_{\theta}+ x_3\partial_{x_1}+x_3\partial_{x_2}-
(x_1+x_2)\partial_{x_3}.\]
So the rotation number is well defined for arbitrary Legendrian
knots in the contact manifold $(S^1\times S^2,\xist )$.
\end{rem}
\section{Fixing a neighbourhood of a Legendrian circle}
We now want to show that after a further contact isotopy we may assume
that $\rc^k\circ f$ fixes a whole neighbourhood
of $K_0$. We formulate this as a general statement.

\begin{lem}
\label{lem:nbhd}
Let $K$ be a Legendrian knot in a contact $3$-manifold $(M,\xi )$,
and let $g$ be a contactomorphism of $(M,\xi )$ that fixes~$K$.
Then $g$ is contact isotopic to a contactomorphism that fixes
a neighbourhood of $K$.
\end{lem}

\begin{proof}
By the tubular neighbourhood theorem for Legendrian
knots~\cite[Cor.~2.5.9]{geig08}, we may identify a tubular neighbourhood
$N(K)$ of $K$ in $(M,\xi )$ with a tubular neighbourhood of
$S^1\times\{ 0\}$ in $S^1\times \R^2=\R /\Z\times\R^2$ with contact structure
$dz-y\, dx=0$; the knot $K$ is identified with $S^1\times\{ 0\}$.

The dilatation
\[ \delta_t(x,y,z):=(x,ty,tz)\]
is a contactomorphism of $S^1\times\R^2$ for each $t\in\R^+$.
This allows us to assume that
the contactomorphic image of $N(K)$ in $S^1\times\R^2$ has been chosen
so large that when we restrict the contactomorphism induced by $g$ to
$S^1\times D^2$, its image will stay inside this image of~$N(K)$
(and that the same holds for the $1$-parameter family of contact embeddings
considered below). In fact, the argument in the proof of
\cite[Prop.~3.1]{cvks09} can be used to show that we may identify
$N(K)$ contactomorphically with all of $S^1\times\R^2$. Although this is not
essential, we shall assume it for ease of notation. Thus, we think of
(the restriction of) $g$ as a contact embedding
\[ g\co S^1\times D^2\longrightarrow S^1\times\R^2.\]
It will suffice to show that $g$ is contact isotopic to the inclusion;
the lemma then follows from the isotopy extension theorem for
contact isotopies, cf.~\cite[Remark~2.6.8]{geig08}.

We now mimic the proof of the contact disc theorem~\cite[Thm.~2.6.7]{geig08}.
Write $g$ in the form
\[ (x,y,z)\longmapsto \bigl( X(x,y,z),Y(x,y,z),Z(x,y,z)\bigr).\]
The condition for this to be a {\em contact\/} embedding is
\[ dZ-Y\, dX=\lambda (dz-y\, dx)\]
with some smooth function $\lambda\co S^1\times D^2\rightarrow\R^+$.
This can be rewritten as the following system of differential
equations:
\[\left\{\begin{array}{rcl}
\displaystyle{\frac{\partial Z}{\partial x}-Y
\frac{\partial X}{\partial x}} & = & -\lambda y,\\
\displaystyle{\frac{\partial Z}{\partial y}-Y
\frac{\partial X}{\partial y}} & = & 0,\rule{0cm}{7mm}\\
\displaystyle{\frac{\partial Z}{\partial z}-Y
\frac{\partial X}{\partial z}} & = & \lambda \rule{0cm}{7mm}.
\end{array}\right.\]
The assumption that $g$ fixes $K$ translates into
\[ X(x,0,0)=x,\;\;\; Y(x,0,0)=0,\;\;\; Z(x,0,0)=0.\]
Now, for $t\in (0,1]$, consider the contact embedding
\[ \delta_t^{-1}\circ g\circ\delta_t (x,y,z)=
\bigl( X(x,ty,tz),\frac{1}{t}Y(x,ty,tz),\frac{1}{t}Z(x,ty,tz)\bigr).\]
For $t\to 0$, this converges to the map
\[ g_0(x,y,z):= \bigl( x, y\cdot\frac{\partial Y}{\partial y}(x,0,0)+
z\cdot \frac{\partial Y}{\partial z}(x,0,0),
y\cdot \frac{\partial Z}{\partial y}(x,0,0)+
z\cdot\frac{\partial Z}{\partial z}(x,0,0)\bigr) .\]
From the above system of differential equations we deduce
\[ \frac{\partial Z}{\partial y}(x,0,0)=0,\;\;\; \frac{\partial Z}{\partial z}
(x,0,0)=\lambda (x,0,0)=:\lambda_0(x).\]
The first differential equation in the above system gives
\[ \frac{\partial^2Z}{\partial z\partial x}-\frac{\partial Y}{\partial z}
\frac{\partial X}{\partial x}-Y\frac{\partial^2X}{\partial z\partial x}=
-\frac{\partial\lambda}{\partial z}\cdot y.\]
When we evaluate this at $(x,0,0)$, we find with the previous equations:
\[ \lambda_0'(x)=\frac{\partial\lambda}{\partial x}(x,0,0)=
\frac{\partial^2Z}{\partial x\partial z}(x,0,0)=
\frac{\partial Y}{\partial z}(x,0,0).\]
Finally, the first differential equation also yields
\[ \frac{\partial^2Z}{\partial y\partial x}-
\frac{\partial Y}{\partial y}\frac{\partial X}{\partial x}
-Y\frac{\partial^2X}{\partial y\partial x}=
-\frac{\partial\lambda}{\partial y}\cdot y -\lambda.\]
Since $\displaystyle{\frac{\partial Z}{\partial y}(x,0,0)=0}$,
we also have $\displaystyle{\frac{\partial^2Z}{\partial x\partial y}
(x,0,0)=0}$.
Thus, evaluating the foregoing equation at $(x,0,0)$ gives
\[ \frac{\partial Y}{\partial y}(x,0,0)=\lambda_0(x).\]

In conclusion, we see that the map $g_0$ takes the form
\[ g_0(x,y,z)=(x,\, y\cdot\lambda_0(x)+z\cdot\lambda_0'(x),\, z\cdot
\lambda_0(x)).\]
It is easy to check that any map of this form (with $\lambda_0\co
S^1\rightarrow\R^+$) is a contact embedding of $S^1\times D^2$ into
$S^1\times\R^2$.

Our initial embedding $g$ is thus seen to be contact isotopic to~$g_0$, and
the convex linear interpolation between $\lambda_0$ and the constant
function~$1$ defines a contact isotopy between $g_0$ and the inclusion
map. This finishes the proof of the lemma.
\end{proof}

\begin{rem}
As explained in \cite[Example 2.5.11]{geig08}, a
universal model for the tubular neighbourhood of a Legendrian
submanifold $L$ in a higher-dimensional contact manifold is
provided by a neighbourhood of the zero section $L\subset T^*L\subset
\R\times T^*L$ in the $1$-jet bundle of $L$ with its canonical
contact structure $dz-\lambda_{\mathrm{can}}=0$, where
$\lambda_{\mathrm{can}}$ is the canonical $1$-form on $T^*L$,
written in local coordinates $\mathbf{q}$ on $L$ and dual coordinates
$\mathbf{p}$ on the fibres of $T^*L$ as $\lambda_{\mathrm{can}}=
\mathbf{p}\, d\mathbf{q}$. The above proof carries over,
{\em mutatis mutandis}, to show that the tubular neighbourhood of a Legendrian
submanifold is unique not only up to contactomorphism, but up
to contact isotopy.
\end{rem}

Here is an alternative proof of Lemma~\ref{lem:nbhd}. Admittedly, the
methods used in it amount to cracking nuts with a sledgehammer,
but they may be of some independent interest.
We define a new contactomorphism $h$ of $(M,\xi )$ as follows,
cf.~\cite[Remark~4.1]{dige09a}.
Choose a standard tubular neighbourhood $N(K)$ of~$K$, where
the contact structure is given by $\cos\theta\, dx-\sin\theta\, dy=0$
under the identification of $N(K)$ with $S^1\times D^2$ (and $K$
with $S^1\times\{ 0\}$). Observe that $N(K)$ may be regarded as the
space of (cooriented) contact elements of~$D^2$.

Set $h=g$ on the closure of $M\setminus N(K)$. On a smaller tubular
neighbourhood $N'\subset N(K)$, set $h=\mbox{\rm id}$.
By the uniqueness up to contactomorphism of the non-rotative tight
contact structure on $T^2\times [0,1]$ with two dividing curves
on each boundary component, see~\cite{hond00}, this $h$ extends
to a contactomorphism on all of $(M,\xi )$.

Then $g^{-1}\circ h$ is a contactomorphism equal to the identity
on $M\setminus N(K)$. So the restriction of $g^{-1}\circ h$ to
$N(K)$ may be regarded as a contactomorphism, equal to the
identity near the boundary,
of the space of contact elements of~$D^2$. According to a result
of Giroux~\cite{giro01}, the group of those contactomorphisms
is connected. This gives a contact isotopy from $g^{-1}\circ h$
to the identity on~$M$. It follows that $g$ is contact isotopic to~$h$,
which has the desired properties.

\begin{rem}
Giroux's paper \cite{giro01} has to be read with a certain amount of
caution. Proposition~10 and the proofs of the main results
(though not the results as such)
are incorrect. The proofs can be fixed using the
methods of~\cite{mass08}.
\end{rem}
\section{Reduction to a space of contact elements}
\label{section:reduction}
In a final step, we want to appeal once more to the result of
Giroux~\cite{giro01} about contactomorphism groups of
spaces of contact elements.

\begin{lem}
The complement of $K_0$ in $(S^1\times S^2,\xist )$ is contactomorphic to
the space of contact elements of~$\,\R^2$.
\end{lem}

\begin{proof}
In \cite{dige07} we described an explicit contactomorphism between
the space of contact elements of $\R^2$ and the complement of a Legendrian
unknot in $S^3$ with its standard contact structure
(which we shall also write as~$\xist$). That complement is
seen to be contactomorphic to $(S^1\times S^2\setminus K_0,\xist )$
as follows.

An alternative surgery picture for $(S^1\times S^2,\xist )$ is
given by a single contact $(+1)$-surgery along a Legendrian unknot in
$(S^3,\xist )$. In this picture,
$K_0$ becomes a Legendrian push-off of the surgery curve,
see~\cite{dige09}. The cancellation lemma from
\cite{dige04}, cf.~\cite[Prop.~6.4.5]{geig08}, says that
contact $(-1)$-surgery along $K_0$ brings us back to $(S^3,\xist )$.
More specifically (as the proof of the cancellation lemma
shows), $K_0$ may be regarded as the belt sphere
of the surgery along the Legendrian unknot in $(S^3,\xist )$,
and the complement of that belt sphere in the surgered
manifold is indeed contactomorphic to the complement of the surgery
curve in the initial manifold.
\end{proof}

In the preceding section we had found an integer $k$ such that
(after a contact isotopy) $\rc^k\circ f$ fixes a
neighbourhood of~$K_0$. So we may interpret this map
as a contactomorphism of the space of contact elements of~$D^2$,
equal to the identity near the boundary. By Giroux~\cite{giro01}, this
contactomorphism is contact isotopic (rel boundary) to the identity.

Thus, in total, our initial diffeomorphism $f$ of $S^1\times S^2$
(acting trivially on homology) has been shown to be isotopic to either
id or~$r$, as was claimed in Proposition~\ref{prop:kernel}.

\begin{rems}
(1) The result of Giroux about the contactomorphism group of the
space of contact elements of $D^2$ uses Cerf's theorem
$\pi_0(\Diff^+(S^3))=0$ in its proof. So the described methods cannot,
as yet, be used to give a contact geometric proof of Cerf's theorem.
However, there is in fact a contact geometric proof of the
slightly weaker form of Cerf's theorem, saying that every diffeomorphism
of $S^3$ extends to a diffeomorphism of the $4$-ball ---
a theorem popularly known as $\Gamma_4=0$. That proof is
due to Eliashberg~\cite{elia92}; for an exposition see~\cite{geig08}.

(2) Observe that our argument has shown the following: any
contactomorphism of $(S^1\times S^2,\xist )$ acting trivially on
homology is contact isotopic to a uniquely determined integer power
of~$\rc$; any contactomorphism that is topologically isotopic to the
identity is contact isotopic to an even power of~$\rc$.
\end{rems}
\section{On the topology of the space of contact structures}
\label{section:space}
Gonzalo and the second author have shown in~\cite{gego04}
that there are essential loops in the space of contact structures
on torus bundles over the circle. The main ingredient in that proof was the
classification of contact structures on the $3$-torus.
Bourgeois~\cite{bour06} reproved their result with the help of
contact homology and used that technique to detect higher
non-trivial homotopy groups of the space of contact structures
on a number of higher-dimensional manifolds. Here we formulate
such a statement for the fundamental group of the space of
contact structures on $S^1\times S^2$.

Write $\Xi_0$ for the component of the space of contact structures
on $S^1\times S^2$ containing~$\xist$, and $\Cont_0$
for the subgroup of $\Diff_0:=\Diff_0(S^1\times S^2)$ consisting of
contactomorphisms of~$\xist$. By Gray stability, we have a surjection
\[ \begin{array}{rccc}
\sigma\co & \Diff_0 & \longrightarrow & \Xi_0\\
          & \phi    & \longmapsto     & T\phi (\xist )
\end{array} \]
with $\sigma^{-1}(\xist )=\Cont_0$.
As shown in \cite{gego04}, this gives rise to a long exact sequence
\[... \stackrel{\Delta}{\longrightarrow} \pi_i(\Cont_0)
\stackrel{\iota_{\#}}{\longrightarrow} \pi_i(\Diff_0)
\stackrel{\sigma_{\#}}{\longrightarrow} \pi_i(\Xi_0)
\stackrel{\Delta}{\longrightarrow} \pi_{i-1}(\Cont_0)
\stackrel{\iota_{\#}}{\longrightarrow} ...,\]
where we write $\iota$ for the inclusion $\Cont_0\rightarrow\Diff_0$;
this is essentially the homotopy long exact sequence of a Serre
fibration.

By the second remark at the end of the preceding section,
we have $\pi_0(\Cont_0)\cong\Z$, generated by the contact isotopy class
of~$\rc^2$. Since this lies in the kernel of~$\iota_{\#}$,
there must be a subgroup isomorphic to $\Z$ in $\pi_1(\Xi_0)$.
If we permit ourselves to rely on some additional information about the
homotopy type of $\Diff_0$, we can actually show this to be the full
fundamental group of~$\Xi_0$.

\begin{prop}
The component $\Xi_0$ of the space of contact structures
on $S^1\times S^2$ containing~$\xist$ has fundamental group
isomorphic to~$\Z$.
\end{prop}

\begin{proof}
The homotopy type of the group of homeomorphisms of $S^1\times S^2$
was determined, modulo the Smale conjecture, by C\'esar de S\'a and
Rourke~\cite{cero79}. Hatcher's proof \cite{hatc83} of the Smale
conjecture not only completes their work, it also implies
--- as shown by Cerf~\cite{cerf59} --- that
the space of diffeomorphisms of any $3$-manifold is homotopy equivalent to
its space of homeomorphisms. Thus,
\[ \Diff_0 (S^1\times S^2)\simeq \SO_2\times\SO_3\times\Omega_0\SO_3,\]
where $\Omega_0\SO_3$ stands for the component of the
contractible loop in the loop space of~$\SO_3$.

Now, $\pi_1(\Omega_0\SO_3)\cong\pi_2(\SO_3)=0$, and the generators
of $\pi_1(\SO_2)$ and $\pi_1(\SO_3)$ in the above factorisation
of $\Diff_0$ can be realised as loops of contactomorphisms
\[ (\theta ,\bfx )\longmapsto (\theta +\varphi ,\bfx ),\;\;
\varphi\in [0,2\pi ], \]
and
\[ (\theta ,\bfx )\longmapsto (\theta ,r_{\varphi}(\bfx )),\;\;
\varphi\in [0,2\pi ], \]
respectively. Thus, the homotopy exact sequence becomes
\[ \pi_1(\Cont_0)\twoheadrightarrow\pi_1(\Diff_0)\rightarrow\pi_1(\Xi_0)
\rightarrow\Z\rightarrow 0.\]
The proposition follows.
\end{proof}
\section{Legendrian knots not distinguished by classical invariants}
\label{section:knots}
In \cite{fras96}, Fraser described an infinite family of Legendrian
knots in the contact manifold
\[ (M_0,\xi_0):=(S^1\times S^2\# S^1\times S^2, \xist\#\xist ), \]
all of which have the same topological knot type and the
same classical invariants $\tb$ and $\rot$, but which are nonetheless
pairwise not Legendrian isotopic. The idea for
distinguishing these knots is to perform Legendrian surgery
on them (or contact $(-1)$-surgery in the language of~\cite{dige04}),
and then to observe that the contact structures on the
surgered manifold (which happens to be the $3$-torus~$T^3$) are
pairwise not isotopic. This argument, in our view, is incomplete because
it hinges on the statement ``Legendrian surgery on Legendrian isotopic
knots produces isotopic contact structures on the surgered manifold'' ---
which is meaningless, as we want to explain.

Suppose you have two Legendrian isotopic knots $L_0$, $L_1$ in a
contact $3$-manifold $(M,\xi )$. The Legendrian isotopy extends to
a contact isotopy $\phi_t$, $t\in [0,1]$, of $(M,\xi )$ with
$\phi_1(L_0)=L_1$. For each $t\in [0,1]$, the contactomorphism
$\phi_t$ of $(M,\xi )$ induces a contactomorphism between the
contact manifold $M_{L_0}$ obtained by  Legendrian surgery
along $L_0$ and the contact manifold $M_{\phi_t(L_0)}$ obtained by
Legendrian surgery along $\phi_t(L_0)$. But there is no way, in general,
to identify $M_{L_0}$ with $M_{\phi_t(L_0)}$ (even as mere differential
manifolds) other than with the diffeomorphism induced by~$\phi_t$.
So we obtain a parametric family of contact manifolds, all of which
are contactomorphic, but not an isotopy of contact structures
on a fixed differential manifold.

In fact, in situations where there is a canonical way of
identifying the surgered manifolds, the statement in question
is false, in general. This is illustrated by the following example
from~\cite[Exercise 11.3.12~(c)]{ozst04}, see
Figure~\ref{figure:twosharks}. Contact $(-1)$-surgery
on the `shark' in $(S^3,\xist )$ with its mouth on the left or on the right
corresponds topologically to a surgery on the unknot with
surgery coefficient $-3$ relative to the surface framing. If we take
the obvious topological identification of the shark with the unknot,
this allows us to identify the surgered manifold in both cases
with the lens space $L(3,1)$. With respect to this identification,
the two resulting contact structures on the surgered manifold
$L(3,1)$ can be distinguished via their induced spin$^c$ structure,
so they are not isotopic.
(Under the identification in question,
which gives the two sharks the same, say the counter-clockwise orientation,
the shark on the left has $\rot =+1$, the one on the right, $\rot =-1$.
This implies that the corresponding spin$^c$ structures have
first Chern class $c_1=\pm1\in H^2(L(3,1);\Z )=\Z_3$;
see~\cite[Prop.~2.3]{gomp98}.)
However, there is a Legendrian isotopy from
one shark to the other (reversing its orientation), and this
induces a contactomorphism of the surgered contact manifolds --- it
is simply the contactomorphism induced by the contactomorphism
$(x,y,z)\mapsto (-x,-y,z)$
relating the two contact surgery diagrams.

\begin{figure}[h]
\labellist
\small\hair 2pt
\pinlabel $-1$ [bl] at 169 125
\pinlabel $-1$ [br] at 508 125
\endlabellist
\centering
\includegraphics[scale=0.3]{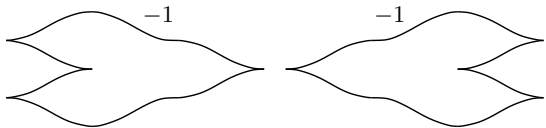}
  \caption{Contactomorphic, non-isotopic contact structures.}
  \label{figure:twosharks}
\end{figure}

Thus, if one wants to show with the help of Legendrian surgery that two
Legendrian knots $L_0,L_1$ cannot be Legendrian isotopic, one has to
require, in general, that the surgered contact manifolds are not
contactomorphic. In Fraser's set-up, unfortunately, the surgered
manifolds happen to be contactomorphic by construction.
Nonetheless, we now want to show that Fraser's idea can be made to work.
In fact, the examples we are going to discuss presently are
explicit realisations of the knots described only implicitly by Fraser.

Figure \ref{figure:fraser} shows a family $L_k$, $k\in\Z$,
of Legendrian knots in $(M_0,\xi_0)$;
for $k<0$, the zig-zags are to be interpreted as $|k|$ pairs of zig-zags
in the opposite direction. Observe that $L_k=\rc^k(L_0)$, where
$\rc$ is regarded as a contactomorphism acting only on the
upper (in the picture) summand $S^1\times S^2$ --- there is a realisation
of $\rc$ that fixes a disc and hence is compatible with
taking the connected sum. All these knots have the
same topological knot type, as can be shown by applying the topological
light bulb trick. Moreover, the well-known formul\ae\ for computing
the classical invariants --- which take the same form for a Legendrian knot
in `standard form' in $(M_0,\xi_0)$ as in $(S^3,\xist )$,
see~\cite{gomp98} --- give $\tb (L_k)=1$ and $\rot (L_k)=0$ for all
$k\in\Z$ (for either orientation of those knots).

\begin{figure}[h]
\labellist
\small\hair 2pt
\pinlabel $\overbrace{\hphantom{xxx}}^k$ [b] at 168 425
\endlabellist
\centering
\includegraphics[scale=0.38]{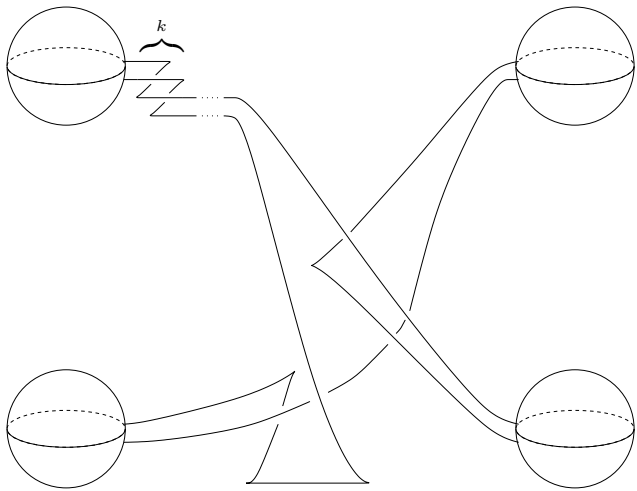}
  \caption{The Legendrian knots $L_k$.}
  \label{figure:fraser}
\end{figure}

\begin{thm}
For $k\neq k'$, the knots $L_k$ and $L_{k'}$ are not Legendrian isotopic.
\end{thm}

\begin{proof}
Arguing by contradiction, let us assume that there are two
Legendrian isotopic knots $L_{k_1}$ and $L_{k_2}$, with $k_1\neq k_2$.
Then $L_0=\rc^{-k_1}(L_{k_1})$ and $L_k=\rc^{-k_1}(L_{k_2})$, where
$k:=k_2-k_1$, will be Legendrian isotopic. Since $L_k=\rc^k(L_0)$,
this implies that $\rc^k$ is contact isotopic to a contactomorphism
of $(M_0,\xi_0)$ that fixes~$L_0$. By Lemma~\ref{lem:nbhd}, $\rc^k$
is then contact isotopic to a contactomorphism $\phi$ that fixes a
neighbourhood $N(L_0)$ of~$L_0$.

The Stein fillable and hence tight contact manifold obtained by
contact $(-1)$-surgery on $L_0$
is~$T^3$, see~\cite[Example~11.2.4]{gost99}, with
its standard contact structure
$\eta_1:=\ker (\sin\theta\, dx-\cos\theta\, dy)$, cf.~\cite{stip02}.
Interpreted as a Kirby diagram, Figure~\ref{figure:fraser}
(with $k=0$ and framing for the handle attachment equal to $-1$
relative to the contact framing) describes $T^2\times D^2$.
The $D^2$-fibre is represented by the cocore of the
$2$-handle, so the belt sphere of the surgery on $L_0$
is an $S^1$-fibre of~$T^3$, see~\cite[Example~4.6.5]{gost99}.

This $S^1$-fibre corresponds (up to isotopy) to the $\theta$-coordinate in the
description of~$\eta_1$, for the $\theta$-circles are uniquely
characterised by the fact that they become homotopically
trivial in any Stein filling $W$ of~$T^3$, see the proof
of~\cite[Lemma~4.3]{stip02}. That fact rests on two observations.
First of all, the homomorphism $H_1(T^3)\rightarrow H_1(W)$
induced by inclusion is
surjective; this follows from the cell structure of Stein manifolds.
Secondly, the $\theta$-fibres must lie in the kernel of this
homomorphism, otherwise one could pass to a cover and obtain
a Stein filling of the contact structure $\eta_n:=
\ker(\sin(n\theta )\, dx-\cos(n\theta )\, dy)$ for some~$n>1$,
which is impossible by a result of Eliashberg~\cite{elia96}.

In the proof of the cancellation lemma given in \cite[p.~323]{geig08}
it is shown explicitly that the belt sphere of the surgery is Legendrian
isotopic, in the surgered manifold, to a Legendrian push-off
of~$L_0$. Alternatively, we can isotope it to a standard Legendrian
meridian of~$L_0$, as shown in \cite[Prop.~2]{dige09}.
We now want to show that this standard Legendrian meridian is
in fact {\em Legendrian} isotopic to the $\theta$-fibre $S^1_{\theta}$
in $(T^3,\eta_1)$. For this we appeal to the classification
of linear Legendrian curves in $(T^3,\eta_1)$ by Ghiggini~\cite{ghig06}.
The Thurston--Bennequin invariant $\tb (L)$ can be defined for such
linear Legendrian curves $L$ as the twisting of the contact structure
relative to an incompressible torus containing~$L$. This means that
$\tb (S^1_{\theta})=-1$, since the contact structure
$\eta_1$ makes one negative twist along a $\theta$-fibre relative
to the framing given by the product structure $T^3=S^1_{\theta}\times
T^2_{x,y}=(\R /2\pi\Z)\times (\R /\Z )^2$
(and the orientation $dy\wedge dx\wedge d\theta$ induced by~$\eta_1$).
This is the maximal $\tb$ in the topological knot type of $S^1_{\theta}$;
see~\cite[Thm.~5.4]{ghig06}. The definition of the rotation number $\rot$
for linear Legendrian curves in $(T^3,\eta_1)$ depends on the choice
of trivialisation of~$\eta_1$, but it can be normalised so that
$\rot =0$ for curves realising the maximal~$\tb$. According
to \cite[Thm.~2.5]{ghig06}, the classical invariants suffice to
classify Legendrian realisations of the topological knot type
of $S^1_{\theta}$. So all we have to show is that the
standard Legendrian meridian $\mu_0$ of $L_0$ has $\tb (\mu_0) =-1$ in 
the surgered manifold, i.e.\ in $(T^3,\eta_1)$.

Now, $\mu_0$ bounds a disc in $S^3$. With respect to the framing
given by that disc in~$S^3$, the contact structure makes one negative twist
along~$\mu_0$ (since $\mu_0$ is a standard Legendrian unknot
in $S^3$ with $\tb (\mu_0) =-1$). A close inspection of
\cite[Example~4.6.5]{gost99} shows that the framing which $\mu_0$
inherits from the meridional disc, now regarded as a framing
of $\mu_0$ in~$T^3$, is the same it inherits from an
incompressible torus, whence it follows that $\tb (\mu_0)=-1$
also in the surgered manifold $(T^3,\eta_1)$. (For that last statement
about framings, imagine $S^3$ being cut in a plane passing through the
attaching balls of one of the $1$-handles, and with the attaching
balls for the second $1$-handle symmetric to this plane. Then, in
the $2$-sphere cut out by this
plane, a circle around one of the attaching balls --- that ball
being seen as a disc in this $2$-sphere ---
defines~$\mu_0$ up to isotopy. The $2$-sphere with two discs
removed, together with a cylinder contained in the boundary
of the $1$-handle, is an incompressible torus
in the surgered manifold, containing~$\mu_0$.)

As a check for consistency, we observe that because of
$\tb (S^1_{\theta})=-1$ in $(T^3,\eta_1)$,
contact $(+1)$-surgery along such a fibre,
which brings us back to $(M_0,\xi_0)$ by the cancellation lemma,
is topologically a surgery with framing given by the
product structure of~$T^3$, and that does indeed produce~$M_0$.

If we perform the surgery along $L_0$ inside the neighbourhood
$N(L_0)$, the fact that the belt sphere of the surgery
is $S^1_{\theta}$ (up to Legendrian isotopy) implies that we
have a contactomorphism between $(M_0\setminus N(L_0),\xi_0)$
and $(T^3\setminus N(S^1_{\theta}),\eta_1)$ for some neighbourhood
$N(S^1_{\theta})$ of~$S^1_{\theta}$.
It follows that the contactomorphism $\phi$ of $(M_0,\xi_0)$,
which fixes $N(L_0)$,
induces a contactomorphism of $(T^3,\eta_1)$ that fixes
$N(S^1_{\theta})$. This may be interpreted as a
contactomorphism of the space of contact elements of $T^2$ with a disc $D^2$
removed, equal to the identity near the boundary. By Giroux's
theorem~\cite{giro01}, this contactomorphism
is contact isotopic (rel boundary) to one that is lifted from
a diffeomorphism of the base $T^2\setminus D^2$.
(Recall that the differential of a diffeomorphism of any given manifold
induces a contactomorphism of the space of contact elements of that
manifold.)
We continue to write $\phi$ for this contactomorphism and its
extension to $(M_0,\xi_0)$.

Using the action of the diffeotopy group of $T^2\setminus D^2$
by contactomorphisms on $(T^3\setminus N(S^1_{\theta}),\eta_1)$,
we may assume that the identification of
$(T^3\setminus N(S^1_{\theta}),\eta_1)$ with $(M_0\setminus N(L_0),\xi_0)$
has been chosen in such a way that one of the standard generators
of $H_1(T^2\setminus D^2)$ corresponds to a loop in $M_0\setminus N(L_0)$
going once (homologically, or geometrically counted with sign)
over the upper $1$-handle in Figure~\ref{figure:fraser}.

A concrete Legendrian realisation $K_1$ of such a loop
is shown in Figure~\ref{figure:3torus},
where $S^1_{\theta}$ is taken to be the fibre over $(x,y)=(1/2,1/2)$.
We take $\oK_1=\{ y=y_0,\,\theta =0\}$ (oriented by $\partial_x$,
cooriented by~$-\partial_y$); its Legendrian lift $K_1$ coincides
with~$\oK_1$.
Transverse to $K_1$ we see an annulus $\{ x=1/2\}$ in $T^3\setminus
N(S^1_{\theta})$. Each of the two boundary
components of that annulus bounds a disc in~$M_0$, so there we have a
$2$-sphere transverse to~$K_1$, corresponding to the $S^2$-factor in the
upper summand $S^1\times S^2$. We also write $K_1$ for the corresponding
Legendrian loop in $(M_0\setminus N(L_0),\xi_0)$.

\begin{figure}[h]
\labellist
\small\hair 2pt
\pinlabel $\oK_1=K_1$ [b] at 126 41
\pinlabel $x=1$ [t] at 362 3
\pinlabel $\theta=2\pi$ [r] at 2 291
\pinlabel $y=1$ [br] at 145 147
\pinlabel $(0,0,0)$ [tr] at 2 3
\endlabellist
\centering
\includegraphics[scale=0.4]{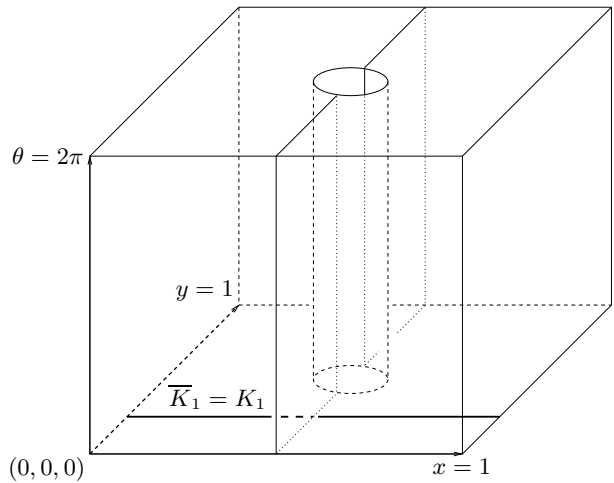}
  \caption{$T^3\setminus N(S^1_{\theta})$.}
  \label{figure:3torus}
\end{figure}

The contactomorphism $\phi$ of $(M_0,\xi_0)$, being isotopic
to~$\rc^k$, sends $K_1$ to a Legendrian knot $\phi (K_1)\subset
M_0\setminus N(L_0)$ that is
smoothly homotopic in $M_0$ to~$K_1$. This translates into
a homotopy in $M_0\setminus N(L_0)$ at the price of adding
a meridional loop every time the original homotopy crosses~$L_0$.
In $T^3\setminus N(S^1_{\theta})$ this becomes a homotopy
between $K_1$ and $\phi (K_1)$, modulo adding a $\theta$-fibre
for each of the meridional crossings. When projected to
$T^2\setminus D^2$, this defines a homotopy between $\oK_1$ and
the projection $\oK_1'$ of $\phi (K_1)$.

If we write $\ophi$ for the diffeomorphism of $T^2\setminus D^2$
whose lift is $\phi$, then $\oK_1'=\ophi (\oK_1)$, so $\oK_1'$ is
a simple closed curve in $T^2\setminus D^2$ homotopic to~$\oK_1$.
By Baer's theorem, see~\cite[6.2.5]{stil93}, there is an isotopy of
$T^2\setminus D^2$, identical near the boundary, that moves $\oK_1'$ back
to~$\oK_1$ (with the original orientation, for homological reasons).
The lift of this isotopy is a contact isotopy
of $(T^3\setminus N(S^1_{\theta}),\eta_1)$, fixed near the boundary, that
moves $\phi (K_1)$ back to~$K_1$.

Thus, $\rc^k$ is contact isotopic to a contactomorphism of $(M_0,\xi_0)$
that fixes $K_1$, which means that $\rc^k$ sends $K_1$ to a Legendrian
isotopic copy of~$K_1$, contradicting the fact that $\rc^k$
changes the rotation number of any oriented Legendrian circle that passes
once (in positive direction, passings counted with sign) over the upper
$1$-handle by~$k$.
\end{proof}

\begin{rem}
Prior to Fraser's work, no examples were known of Legendrian knots
that could not be distinguished by the classical
invariants. The first examples of this type in $(\R^3,\xist )$
were found by Chekanov~\cite{chek02},
who used Legendrian contact homology to distinguish
the knots. Various other non-classical invariants have been
developed in the meantime, such as normal rulings~\cite{chpu05,fuch03}
or knot Floer homology invariants~\cite{loss,ost08}.
\end{rem}

\begin{ack}
We thank Yasha Eliashberg and Andr\'as Stipsicz for useful conversations,
and Bijan Sahamie for his critical reading of an earlier
manuscript and several constructive comments.

F.~D.\ is partially supported by
grant no.\ 10631060 of the National Natural Science Foundation
of China.

The final writing was done while H.~G.\ was a guest of Peking University,
Beijing.
\end{ack}

\end{document}